\documentclass{article}
\usepackage{graphicx} 
\usepackage{authblk}
\usepackage{biblatex} 
\usepackage{amsmath}
\usepackage{amsthm}
\usepackage{amsfonts}
\usepackage[colorlinks,allcolors=blue]{hyperref}
\theoremstyle{plain}
\newtheorem{theorem}{Theorem}
\newtheorem*{theorem*}{Theorem}
\newtheorem*{lemma*}{Lemma}
\newtheorem{lemma}{Lemma}

\theoremstyle{definition}
\newtheorem{definition}{Definition}

\theoremstyle{remark}

\title{Improving cardinality estimation of sums of sets with convexity}
\author{Jun Ikeda\footnote{jikeda@caltech.edu}}

\affil{California Institute of Technology, Pasadena, California, 91125, United States of America}

\begin{document}

\maketitle

\begin{abstract}
This paper presents a slight improvement of the estimate of sumsets of convex sets with negative discrete third derivative. The proposed method is based on some previous works in incidence geometry and use of spectrum method developed earlier for notable progresses in this fields. 
\end{abstract}

\section{\centering INTRODUCTION}
Let $(G, +)$ be an abelian group. For two finite subsets $A, B \subset G$, their sumset is defined as follows:
\begin{equation*}
    A+B = \{ a+b; (a, b) \in A \times B \}
\end{equation*}
Henceforward, $G$ is assumed to be a totally ordered abelian group. 

It is conjectured by Erd\"{o}s and Hegyv\'{a}ri that $|A+A| \gg_{\epsilon} |A|^{2-\epsilon}$ for convex sets $A$.
\begin{definition}
    Let $A=\{a_1 < \cdots < a_{|A|}\} \subset G$. $A$ is a convex set if for all $i$ 
    \begin{equation*}
        a_{i+2} - a_{i+1} > a_{i+1} - a_{i}.
    \end{equation*}
\end{definition}
The most classical lower bound, $|A+A| \gg |A|^{3/2}$, is given by Garaev with a very elementary geometric observation of convexity \cite{garaev2000lower}.
This has been improved a number of times using observations
on related combinatorial quantities such as $E (A)$ \cite{li2011theoremschoenshkredovsumsets}, \cite{schoen2011sumsets}, \cite{shkredov2012somenewineq}, \cite{shkredov2013new}, \cite{Shkredov2015OnSO}, \cite{Olmezov2020SharpeningAE}. 
\begin{definition}
    Let $A \subset G$. The second additive energy is
    \begin{align*}
        E(A, B) &= \#\{(a_1, b_1, a_2, b_2) \in (A \times B)^{2}; \\
        &\qquad\qquad a_1 - b_1 = a_2 - b_2 \} .
    \end{align*}
    Hereafter, $E(A)$ denotes $E(A, A)$.
\end{definition}

The current best record, $|A+A| \gg |A|^{\frac{30}{19}}$, is by Rudnev and Stevens, who carefully avoided $E(A)$, for which 
a sharp bound is yet to be proved \cite{Rudnev2020AnUO}.

It was noticed by Ol'mezov that for a subset of all convex sets, this can be significantly improved \cite{Olmezov2020AdditivePO}. 
\begin{definition}
    Let $A \subset G$. $A$ can be arranged in increasing order: $A = \left\{a_1 < a_2 < \cdots < a_{n}\right\}$. Then define (higher) derivative as follows:
    \begin{align*}
        \Delta A &= \left\{a_{i}^{\prime} = a_{i+1} - a_{i}\right\}, \\
        \Delta^{(k)} A &= \left\{a_{i}^{(k)} = a_{i+1}^{(k-1)} - a_{i}^{(k-1)}\right\}. \\
    \end{align*}
    $k$'th derivative of $A$ is positive if $a_{i}^{(k)} > 0$ for all $i$, and it is denoted as $\Delta^{(k)} A > 0$.
\end{definition}

Using this terminology, $A$ is convex
if and only if $\Delta^2 A > 0$.

\begin{theorem}[Ol’mezov]
    Let $A \subset G$. If $\Delta^2 A > 0$ and $\Delta^3 A \le 0$,
    \begin{equation}\label{olmezov negative third result}
        |A+A| \gg |A|^{\frac{8}{5}}.
    \end{equation}
    If $\Delta^2 A > 0, \Delta^3 A < 0,$ and $\Delta^4 A \le 0$,
    \begin{equation}\label{olmezov negative forth result}
        |A+A| \gg |A|^{\frac{5}{3}}.
    \end{equation}
\end{theorem}

In this paper, we show a slightly better bound for convex sets with negative third derivative. 
\begin{theorem}\label{main theorem}
    Let $A \subset G$. If $\Delta^2 A > 0$ and $\Delta^3 A < 0$,
    \begin{equation}
        |A+A| \gg |A|^{\frac{221}{137}}.
    \end{equation}
\end{theorem}
Although the third derivative has to be strictly negative unlike in (\ref{olmezov negative third result}), $221/137 \sim 1.61...$ is greater than $8/5$ by $0.01..$. This result was obtained by a fundamental inequality discovered by Ol'mezov in \cite{Olmezov2020AdditivePO}, Shkredov's spectrum method, and some inequalities involving higher discrete derivatives of sets given by Bradshaw, Hanson, and Rudnev \cite{Bradshaw2021HigherCA}.

In this paper, $\gg$ denotes the Vinogradov’s symbols: $X \gg Y$ if $\exists c > 0$ such that $X \ge cY$. 
By a slight abuse of notation we often ignore $\log$ terms and contain it within the Vinogradov's symbols. 
That is, we denote $X \gg Y$ meaning that $X \gg_d Y \log^d{|Y|}$ for some constant $d$. 

\section{\centering PRELIMINARY}
In this section, we provide some preliminary definitions and results.
First, let us introduce some more definitions and notations that are useful to state the lemmas.

\begin{definition}
    Let $f, g: G \rightarrow \mathbb{C}$. Then, the (difference/sum) convolutions are defined as follows.
    \begin{align*}
        f \circ g (x) &= \sum_{s} f(s)g(s+x), \\
        f \ast g (x) &= \sum_{s} f(s)g(x-s) .
    \end{align*}
\end{definition}

The following are natural generalizations of $E(A, B)$.
\begin{definition}
    Let $A, B \subset G$ and let $A$ denote its characteristic function. 
    The $k$'th additive energy, $E_{k}(A, B)$, between $A, B$ is
    \begin{align*}
        E_k(A, B) &= \#\{(a_1, b_1, a_2, b_2, ..., a_k, b_k) \in (A \times B)^{k}; \\
        &\qquad\qquad a_1 - b_1 = a_2 - b_2 = \cdots = a_k - b_k\} \\
        &= \sum_x (A \circ B (x))^k.
    \end{align*}
    Hereafter, $E_k(A)$ denotes $E_k(A, A)$ and $E(A, B)$ denotes $E_2(A, B)$. \\
    Let $A_1, ... A_k \subset G$. The $T$-additive energy, $T(A_1, ... A_k)$, is
    \begin{align*}
        T(A_1, ... A_k) &= \#\{(a_1, ..., a_k, a_1^{\prime}, ..., a_k^{\prime}) \in (A_1 \times \cdots \times A_k)^2; \\
        &\qquad\qquad a_1 + \cdots + a_k = a_1^{\prime} + \cdots + a_k^{\prime}\} \\
        &= \sum_x (A_1 \ast \cdots \ast A_k (x))^2.
    \end{align*}
    Similarly, $T_k(A)$ denotes $T(A, ..., A)$, $T$-additive energy between $k$-many copies of $A$.
\end{definition}

Recognize that $E(A) = \#\{(a_1, a'_1, a_2, a'_2) \in (A \times A)^{2}; a_1 - a'_1 = a_2 - a'_2 \} = \#\{(a_1, a'_1, a_2, a'_2) \in (A \times A)^{2}; a_1 + a'_2 = a_2 + a'_1 \} = T_2(A)$.

Recall the bound of $E(A, B)$ for convex $A$ and any $B$ given by Li \cite{li2011theoremschoenshkredovsumsets}.
\begin{lemma}[Li]\label{li E(A,B) A B3/2}
    Let $A, B \subset G$. If $A$ is convex,
    \begin{equation}
        E(A, B) \ll |A| |B|^{3/2}.
    \end{equation}
\end{lemma}
Due to its generality, this bound proved to be useful. 
However, when $A=B$, Shkredov noted that the upper bound can be improved
using the spectrum method. 
This method makes use of the following positive definite operators and inequalities on their principle eigenvalues \cite{shkredov2012somenewineq}, \cite{shkredov2013new}, \cite{Shkredov2015OnSO}. 
\begin{definition}
    Let $g: G \rightarrow \mathbb{C}$ and $A, B \subset G$. Define operators as
    \begin{align*}
        T_{A, B}^{g}(x, y) &= A(x)B(y)g(x-y), \\
        \tilde{T}_{A, B}^{g}(x, y) &= A(x)B(y)g(x+y).
    \end{align*}
\end{definition}
In \cite{Shkredov2015OnSO}, for example, Shkredov used $g = A\circ A$ and an inequality obtained from Rayleigh quotient to show
\begin{align}\label{spectrumexample}
    \frac{|A|^{10}}{|A+A|^2} \ll E_3(A) E_3(A, A, S_1),
\end{align}
where $S_1 = \{x;~ \tau \le A \ast A (x) \le 2\tau \}$ such that $\tau \gg |A|^2/|A+A|$ and
$|A|^2 \ll \sum_{x \in S_1} (A \ast A)^2 (x)$,
and $E_3(A, A, S_1) = \sum_x (A \circ A)(x) (A \circ A)(x) (S_1 \circ S_1)(x)$.
Such $\tau$ is shown to exist using the dyadic pigeonhole principle. 

It is worth mentioning Ol'mezov proposed a graph theoretic interpretation of (\ref{spectrumexample}).
He discovered that this type of inequalities can be derived from partially proved Sidorenko's conjecture, 
which bounds the number of bipartite subgraphs \cite{OlmezovElementaryApproach}.

Ol'mezov showed that negative third/forth derivative makes it easier to prove higher bounds \cite{Olmezov2020AdditivePO}. To prove (\ref{olmezov negative third result}) and (\ref{olmezov negative forth result}), he made some geometric observation analogous to Garaev's proof. 
\begin{lemma}[Ol'mezov]
    Let $A \subset G$ such that $\Delta^2 A > 0, \Delta^3 A \le 0$. 
    Sort the elements $\{s_i; i \in \mathbb{N}\} = A-A,~A+A$ such that $A \circ A (s_{i}) \ge A \circ A (s_{i+1}), ~ A \ast A (s_{i}) \ge A \ast A (s_{i+1})$ respectively. Then,
    \begin{align}\label{Ol'mezov negative third fundamental}
        A \circ A (s_j) &\ll |A|j^{-\frac{3}{8}}, \\
        A \ast A (s_j) &\ll |A|j^{-\frac{3}{8}}.
    \end{align}
\end{lemma}

For sets with higher derivative, Bradshaw et al. also discovered a sharper inequality of $T$-additive energy \cite{Bradshaw2021HigherCA}. 
\begin{lemma}[Bradshaw et al.]\label{higher convexity T energy}
    Let $A = \left\{a_1 < a_2 < \cdots < a_{n}\right\} \subset G$ be a convex set such that $\forall i \neq j \in \{1, ..., n-2\}, a_{i+2} - 2a_{i+1} + a_{i} \neq a_{j+2} - 2a_{j+1} + a_{j}$. Then,
    \begin{equation}
        T_4(A) \ll |A|^3T_2(A) + |A|^4T_2^{\frac{3}{4}}(A),
    \end{equation}
\end{lemma}
In \cite{Bradshaw2021HigherCA}, the main focus is on $A$ such that $\Delta^3 A > 0$, but it is also mentioned that the condition can be relaxed as
stated above, which enables us to use it for sets with $\Delta A < 0$.

\section{\centering PROOF OF MAIN THEOREM}
We first prove a new bound for $E(A, S)$ and use it to prove a slightly better upper bound for 
the second energy, and lower bound for the sumset.
\begin{lemma}\label{E(A,S) bound}
    Let $A \subset G$ be a convex set such that $\Delta^3 A < 0$. Let \\
    $S_{\tau} = \left\{s \in A-A; ~ A \circ A (s) \ge \tau\right\}$. Then,
    \begin{equation}
        E(A, S_{\tau}) \ll \tau^{-2} |A|^2 E^{\frac{7}{8}}(A).
    \end{equation}
\end{lemma}
\begin{proof}
    By Cauchy-Schwartz inequality, the left side is bounded by $E^{\frac{1}{2}}(A) E^{\frac{1}{2}}(S_{\tau})$. 
    For $E(S_{\tau})$, one has
    \begin{align*}
        E(S_{\tau}) 
        &\le \tau^{-4} \#\{(a_1, b_1, a'_1, b'_1, a_2, b_2, a'_2, b'_2) \in A^8; \\
        &\qquad\qquad a_1 - b_1 + a'_1 - b'_1 = a_2 - b_2 + a'_2 - b'_2\} \\
        &= \tau^{-4} T_4(A),
    \end{align*}
    where the inequality follows by definition of $S_{\tau}$.
    
    Futhermore, using Lemma \ref{higher convexity T energy}, 
    \begin{align}
        E(S_{\tau}) \le \tau^{-4}T_4(A) \ll \tau^{-4}|A|^4T_2^{\frac{3}{4}}(A).
    \end{align}
\end{proof}
Using Lemma \ref{E(A,S) bound} in combination with the spectrum method, Ol'mezov's bound for the second energy, $|A|^{\frac{12}{5}}$ can be improved.
\begin{lemma}\label{E(A) new bound}
    Let $A \subset G$ be a convex set with negative third derivative. Then,
    \begin{equation}
        E(A) \ll |A|^{\frac{328}{137}}.
    \end{equation}
\end{lemma}
\begin{proof}    
    By (\ref{Ol'mezov negative third fundamental}),
    \begin{equation*}
        \sum_{\frac{|A|^4}{E^{3/2}(A)} \ll A \circ A (x)} (A \circ A)^2 (x)
        \ll |A|^2 \sum_{j=1}^{c \frac{E^4(A)}{|A|^8}} j^{-3/4} 
        \ll E(A),
    \end{equation*}
    where $c > 0$ is some constant. 
    On the other hand, $\sum_{A \circ A (x) \ll \frac{E(A)}{|A|^2}} (A \circ A)^2 (x) \ll E(A)$. 
    Therefore, 
    By pigeonhole principle, there exists $\Delta$ such that $\frac{E(A)}{|A|^2} \ll \Delta \ll \frac{|A|^4}{E^{3/2}(A)}$ and $\sum_{\Delta \le A \circ A (x) \le 2\Delta} (A \circ A (x))^ 2 \gg E(A)$,
    where the factor of $\log_2 |A| $ is contained in the Vinogradov's symbol.
    
    Let $S_{\Delta} = \{x; \Delta \le A \circ A (x) \le 2\Delta\}$. 
    By considering $T_{A, A}^{(A \circ A) \cdot S_{\Delta}} (x, y)$ and following a similar derivation to \cite{Shkredov2015OnSO}, \cite{Olmezov2020SharpeningAE},
    we obtain
    \begin{equation}
        \frac{E^6(A)}{|A|^6} \ll E_3(A) \Delta^3 \tau^2 E^{\frac{1}{2}}(A, S_{\Delta}) E^{\frac{1}{2}}(A, S_{\tau}).
    \end{equation}
    Applying Lemma \ref{li E(A,B) A B3/2} on $E^{\frac{1}{2}}(A, S_{\tau})$ and Lemma \ref{E(A,S) bound} on $E^{\frac{1}{2}}(A, S_{\Delta})$, and using the fact that $\tau^{k}|S_{\tau}| \le E_k(A)$,

    \begin{equation}
        \frac{E^6(A)}{|A|^6} \ll E_3(A) \Delta^2 E_{8/3}^{\frac{3}{4}}(A) |A|^{\frac{3}{2}} E^{\frac{7}{16}}(A).
    \end{equation}
    The bounds of $\Delta$ and (\ref{Ol'mezov negative third fundamental}) gives the required inequality.
\end{proof}

With this new bound, the sumset of such sets is also sharpened.
\begin{proof}[Proof of Theorem \ref{main theorem}]
    Similarly to the proof of Lemma \ref{E(A) new bound}, $\sum_{\frac{|A|^2}{|A+A|} \ll A \circ A (x) \ll |A|^{\frac{2}{5}}} (A \ast A)^2 (x) \gg |A|^2$,
    and by pigeonhole principle, there exists $\Delta$ such that $\frac{|A|^2}{|A+A|} \ll \Delta \ll |A|^{\frac{2}{5}}$ and $\sum_{\Delta \le A \ast A (x) \le 2\Delta} (A \ast A (x))^ 2 \gg |A|^2$
    up to a factor of $\log |A|$.
    Let $S_{\Delta} = \{x; \Delta \le A \ast A (x) \le 2\Delta\}$. 
    By considering $T_{A, A}^{S_{\Delta}} (x, y)$ and following a similar derivation to \cite{Shkredov2015OnSO}, \cite{Olmezov2020SharpeningAE},
    we obtain
    \begin{equation}
        \frac{|A|^{10}}{|A+A|^2} \ll E_3(A) \tau^2 \Delta^{-1} E^{\frac{1}{2}}(A, S_{\Delta}) E^{\frac{1}{2}}(A, S_{\tau}).
    \end{equation}
    Applying Lemma \ref{li E(A,B) A B3/2} on $E^{\frac{1}{2}}(A, S_{\tau})$ and Lemma \ref{E(A,S) bound} on $E^{\frac{1}{2}}(A, S_{\Delta})$, and using the fact that $\tau^{k}|S_{\tau}| \le E_k(A)$,
    \begin{equation}
        \frac{|A|^{10}}{|A+A|^2} \ll E_3(A) \Delta^{-2} E_{8/3}^{\frac{3}{4}}(A) |A|^{\frac{3}{2}} E^{\frac{7}{16}}(A).
    \end{equation}
    The bounds of $\Delta$ and (\ref{Ol'mezov negative third fundamental}) and Lemma \ref{E(A) new bound} gives the required inequality.
\end{proof}

\section{\centering ACKKNOWLEDGEMENTS}
The author would like to express gratitude to Prof. Nets H. Katz and Prof. Shukun Wu for their invaluable guidance, insightful discussions, and constructive suggestions throughout this research. Appreciation is also extended to the Student Faculty Program at Caltech and its donors for providing the funding necessary to carry out this study.

Finally, the author wishes to thank the anonymous reviewers for their thoughtful comments and recommendations, which significantly contributed to improving the quality of this work.

\printbibliography 

\end{document}